\RequirePackage{fix-cm}
\documentclass[smallextended]{svjour3}       
\smartqed  

\usepackage{amsmath,amscd, amsfonts, amssymb, graphicx, color,bbold,amsfonts,adjustbox,url}
\usepackage{tikz,tikz-cd}
\usepackage[bookmarksnumbered, colorlinks, plainpages]{hyperref}
\usepackage{url}
\usetikzlibrary{matrix,arrows,decorations.pathmorphing}
\usepackage[utf8]{inputenc}
\usepackage[english]{babel}
\begin{document}

\title{Some functorial properties of Schatten classes.\thanks{The research work presented here  was jointly carried out by the authors while being funded by Council of Scientific and Industrial Research- Government of India  and University Grants Commission-Government of India.}}


\author{Lav Kumar Singh        \and
        A.K. Verma 
}


\institute{Lav Kumar Singh \at
             +918448746990\\              
              \email{\url{lavksingh@hotmail.com; lav17_sps@jnu.ac.in}}            \\
             \emph{School of Physical Sciences, J.N.U, New Delhi}   
           \and
           Arvind Kumar Verma \at
           +919425852875\\
             \email{\url{arvindamathematica@gmail.com}}\\
              \emph{Dr. Harisingh Gour Central University, Sagar(M.P)-India}
}

\date{Received: date / Accepted: date}

\maketitle

\begin{abstract}
In this paper, we begin with the study of elements in $C^*$-algebras which are mapped to Schatten class ideals through the faithful left regular representation. We further give some functorial properties of Schatten classes on the category of representations of a  $C^*$-algebra and category of unitary representations of a group. 
\end{abstract}
\section{Introduction}
\label{intro}
We aim to explore how Schatten classes on Hilbert spaces influence the representations of $C^*$-algebras or locally compact Hausdorff groups. We begin our study in the second section with the left regular representation of $L^\infty(X)$ on $L^2(X)$, where $(X,\mathcal B(X),\mu)$ is a $\sigma$-finite measure space. We give a few conditions on measure properties of $X$ such that the pullback of $p^{th}$-Schatten ideal is not trivial (zero). This establishes the fact that these pullbacks are not trivially zero. Then in the third section, we give a nice functor on category of representations of a $C^*$-algebra and category of representations of a group which is associated to Schatten classes. Let us begin with a few preliminaries.
\subsection{Measure space}
\label{intro 1}
Let $(X,\mu)$ be a measure space. A set $A\subset X$ is said to be {\bf atom} of $\mu$ if $\mu(A)>0$ and for every $B\subset A$, either $\mu(B)=0$ or $\mu(B)=\mu(A)$. Subsets of $X$ which are not atoms are said to be {\bf diffuse}. A measure is said to be atomic if there exists atleast one atom, else the measure is said to be without atom.
\subsection{Schatten classes on Hilbert spaces }
\label{intro 2} 
Let $\mathcal H$ be Hilbert space and $T\in B(\mathcal H)$. For $p\in[1,\infty)$ define Schatten p-norm of $T$ as $$||T||_p=\sum_{n\geq 1}(s_n^p(T))^{1/p},$$ where $s_1(T)\geq s_2(T)\geq\hdots s_n(T)\geq 0$ are the singular values of operator $|T|$. Clearly $||T||_p=\text{tr}(|T|^p)$. The \textbf{p-th Schatten class operators} are those $T\in B(\mathcal H)$ for which $||T||_p <\infty$. Let $S_p(\mathcal H)$ denotes the collection of all p-th Schatten class operators in $B(\mathcal H)$. Then $S_p(\mathcal H)$ is an ideal of $B(\mathcal H)$ and $||\cdot||_p$ is a norm on $S_p(\mathcal H)$ which makes it a Banach $*$-algebra (involutions are adjoint maps). Further, we have the containment $S_p(\mathcal H)\subset S_q(\mathcal H)$ for each $q>p$. The Schatten p-norm satisfies $||\cdot||_q\leq ||\cdot ||_p$ for all $q\geq p$.
Operators in $S_p(\mathcal H)$ are compact, and $S_p(\mathcal H)$ contains all finite rank operators.  $S_p(\mathcal H)$ is reflexive for $1<p<\infty$ since they are uniformly convex. We also know that $S_p(\mathcal H)^*=S_q(\mathcal H)$, where $\frac{1}{p}+\frac{1}{q}=1$. The ideal $S_2(\mathcal H)$ is a Hilbert space with inner product $\left<A,B\right>=\operatorname{Tr}(B^*A)$. A detailed discussion about these facts can be found in \cite{Palmer}. We will use the notation $S_\infty(\mathcal H)=B_0(\mathcal H)$ for the algebra of compact operators.
\subsection{Representations of $C^*$ algebras and locally compact groups.} 
\label{intro 3}
A representation  of a $C^*$-algebra $\mathcal A$ is a continuous $*$-homomorphism $\mathcal A \to B(\mathcal H)$ for some Hilbert space $\mathcal H$. Injective representations are called \textbf{faithful} representations. Gelfand Naimark theorem states that every $C^*$-algebra possess a faithful representation on some Hilbert space.\cite[Ch.3]{Murphy}\\ 
A \textbf{unitary representation} of a locally compact Hausdorff group $G$ is a homomorphism $G\to U(\mathcal H)$ for some Hilbert space $H$, continuous with respect to strong topology on $U(\mathcal H)$, where $U(\mathcal H)$ denotes the group of unitary operators on $\mathcal H$ \cite[Ch.3]{Folland}.\\

\subsection{Categories}
\label{intro 4}
By $\bf Ban_1$, we denote the category whose objects are Banach spaces and morphisms are contractive linear maps.  $\mathcal Rep(\mathcal A)$ denotes the category whose objects are representations of $C^*$-algebra $\mathcal A$ and natural transformations generated by intertwining operators define morphisms, i.e. $\phi:\pi_1\to\pi_2$ is a morphism between two representations $\pi_1:\mathcal A\to B(\mathcal H_1)$ and $\pi_2:\mathcal A\to B(\mathcal H_2)$, then there exists an operator $U:\mathcal H_1\to\mathcal H_2$ such that for each $x\in \mathcal A$ the  diagram in fig. \ref{Fig1} commutes.\\

\begin{figure}
	\begin{tikzcd}
	\mathcal H_1 \arrow[d, "U"] \arrow[rr, "\pi_1(x)"] &  & \mathcal H_1 \arrow[d, "U"] \\
	\mathcal H_2 \arrow[rr, "\pi_2(x)"]                &  & \mathcal H_2               
	\end{tikzcd}
	\caption{}\label{Fig1}
\end{figure}
Similarly, we denote the category of unitary representations of a group $G$ by $\mathcal Rep(G)$ where natural transformations generated by intertwining operators define morphisms.\\
An assignment $\mathcal F:\mathcal C\to\mathcal D$, where $\mathcal C, \mathcal D$ are any two categories, is said to be a functor if it assigns to each object $C\in \mathcal C$ a unique object $\mathcal F(C)\in\mathcal D$ and to  each morphism $f:C_1\to C_2$ in $\mathcal C$ a unique morphism $\mathcal F(f):\mathcal F(C_1)\to\mathcal F(C_2)$ in $\mathcal D$ such that $\mathcal F(\operatorname{Id_C})=\operatorname{Id}_{\mathcal F(C)}$ and $\mathcal F(g\circ f)=\mathcal F(g)\circ \mathcal F(f)$ for any pair of morphisms $f,g$ in $\mathcal C$ (given that they can be composed).
\section{Schatten type functions in $L^\infty(X)$.}
\label{sec:1}
Let $X$ be a measure space with $\sigma$-finite Borel measure $\mu$. Consider the faithful representation of $C^*$-algebra $L^\infty(X)$ by $$\pi:L^\infty(X)\to B(L^2(X))$$ such that $\pi(f)g=f.g$ for each $f\in L^\infty(X)$ and $g\in L^2(X)$. For each $1\leq p\leq \infty$, we define the set $$ S_p(X)=\pi^{-1}\left( S_p(L^2(X))\right),$$ which is a collection of all the functions in $L^\infty(X)$ which are mapped to Schatten $p$-class operators on Hilbert space $L^2(X)$. Clearly, $S_p(X)$ is an ideal in $L^\infty(X)$.
\begin{lemma}
	For group $\mathbb{Z}$ with Haar measure(counting measure), $S_p(\mathbb Z)=\ell^p(\mathbb Z)$ for all $1\leq p\leq \infty$.
\end{lemma} 
\begin{proof}
	Consider the standard orthonormal basis $\{e_i\}_{i\in \mathbb Z}$ of $\ell^2(\mathbb Z)$. Let $f\in \ell^\infty$. Then $f\in S_p(\mathbb Z)$ if and only if $$||\pi(f)||_p=\sum_{n\in\mathbb Z}\left<|f(n)|^pe_n,e_n\right>=\sum_{n\in \mathbb Z}|f(n)|^p<\infty$$
	Hence $S_p(\mathbb Z)=\ell^p(\mathbb Z)$ for all $1\leq p<\infty$. Case $p=\infty$ is trivial.
\end{proof}
\begin{remark}\label{Rem1}The collection $S_p(X)$ is not independent of the faithful representation. For example, consider the faithful representation $\pi\otimes \mathbb 1:\ell^2(\mathbb Z)\to B(\ell^2(\mathbb Z)\otimes \ell^2(\mathbb Z))$. In this case, none of the sequences is mapped to trace class operator.
\end{remark}
\begin{lemma}
	For group $\mathbb R$ with Haar measure, $S_p(\mathbb R)=\{0\}$ for all $1\leq p< \infty$ and $S_\infty(\mathbb R)=L^\infty(\mathbb R)$. Even for compact group $\mathbb S^1$ with Haar measure, $S_p(\mathbb S^1)=\{0\}$ for $1\leq p <\infty$ and $S_\infty(\mathbb S^1)=L^\infty(\mathbb S^1)$.
\end{lemma}
\begin{proof}
~~	Consider the Gabor orthonormal basis $\{e_{n,m}\}_{n,m}$ of $L^2(\mathbb R)$, where  $e_{n,m}(x)=e^{2\pi imx}\chi_{[n,n+1]}$ . Let $f\in L^\infty(\mathbb R)$ and $1\leq p<\infty$. Then $f\in S_p(\mathbb R)$ if and only if $$||\pi(f)||_p=\sum_{n,m\in\mathbb Z}\int_{\mathbb R} |f|^pe_{n,m}.\overline{e_{n,m}}d\mu=\sum_{n,m}\int_n^{n+1}|f|^pd\mu<\infty .$$
	\noindent
	This is unconditional sum of positive numbers which converges if and only if $\int_{\mathbb R}|f|^pd\mu=0$ and hence if and only if $f=0$. \\
	Similarly, using the orthonormal basis $\{e^{2\pi inx}\}_{n\in\mathbb Z}$ for $L^2(\mathbb S^1)$, one can prove that $f\in S_p(\mathbb S^1)$ if and only if $f=0$. The $p=\infty$ situation is easily seen to be true in both cases.
\end{proof}

\begin{remark} The above examples may lead us to the hypothesis that $S_p(X)$ is non-trivial if and only if $X$ is discrete, which is not true as proved below.
\end{remark}
\begin{theorem}\label{Th.1}
	For a measure space $X$ with a $\sigma$-finite  Borel measure $\mu$, the ideal $S_p(X)$ for $1\leq p <\infty$ is trivial(zero) if and only $\mu$ is without atom. \footnote[1]{This proof  is motivated by the idea provided by Prof. Martin Argerami.}
\end{theorem}
\begin{proof}
	Suppose $(X,\mu)$ is atomic. Then we have $X=F\cup(\cup_{x\in D}\{x\})$, where $F$ is the diffuse set(without atoms) and $D$ is the set of all atoms (atoms in regular Borel measure are always almost singleton and countable\cite[2.2]{Chung}). Let $f=\sum_{x\in D} \beta_x \chi_{x}$ such that $\sum_{x\in D}|\beta_x|^p<\infty$. Notice that the set $\left\{\frac{1}{\mu({x})^{1/2}}\chi_{x}\right\}$ is orthonormal . Hence, $$||\pi(f)||_p=\sum_{x\in D}|\beta_x|^p<\infty.$$
	It follows that $f\in S_p(X)$ and \begin{equation} S_p(X)=\{f\in L^\infty(X) :\operatorname{supp(f)}\subset D \text{ and } \sum_{x\in D}|f(x)|^p<\infty\}\end{equation} for all $1\leq p<\infty$.
	Conversely, $S_p(X)$ is non-empty for all $1\leq p<\infty$. Let $f\in S_p(X)$ for a fixed $p$. Then $\pi(f)$ is compact(all Schatten class operators are compact) and $\sigma(\pi(f))=\overline{f(X)}$ must be finite or a set with limit point $0$. So $f$ must be a simple function, $f=\sum_j\alpha_j\chi_{E_j}$. Here projection $\pi(\chi_{E_j})$ must be finite rank except when $\alpha_j=0$. Now, if the set $D$ of atoms is empty then $\mu(E_j\cap F)>0$ for every single $j$ and this will force that $\chi_{E_j}$ is not finite rank. Thus, $D$ must be non-empty.  
\end{proof}

\begin{remark} Notice that $S_p(\mathbb R)$ and $S_p(\mathbb S^1)$ were empty for $1\leq p<\infty$ because the Haar measure on $\mathbb R$ and $\mathbb S^1$ are non-atomic. $S_p(\mathbb Z)$ was non-empty because $\mathbb Z$ with counting measure is atomic (singletons are atoms). There are in fact compact non-discrete groups which are atomic. For example, consider the infinite product $\Pi_J G$ with counting measure, where $G$ is a finite group with discrete topology. Clearly, $\Pi_J G$ is a non-discrete compact group which is atomic (singletons are the atoms, and diffuse is empty). For groups with left/right invariant measures, even stronger result holds.
\end{remark}
\begin{theorem}
	Let $G$ be a group equipped with left/right invariant $\sigma$-finite Borel measure $\mu$. Then $S_p(G)$ for $1\leq p <\infty$ is non-empty if and only if $\mu$ is counting measure (up to product with scalars)  and $G$ is a countable group.
\end{theorem}
\begin{proof}
	From Theorem 1.  we know that $S_p(G)$ is non-empty if and only if $\mu$ is atomic. Let $\{x\}$ be any atom. Then $\mu(x)>0$. By left/right translation invariance we know that $\mu(\{g\})=\mu(\{x\})$ for all $g\in G$. Hence each singleton in $G$ is an atom. However, atoms in a $\sigma$-finite measure space are countable \cite[2.2,ex.11]{Chung}. Thus $G$ must be a countable group in which all singletons are atoms of the same measure. Hence $\mu$ must be some scalar multiple of counting measure. 
\end{proof}
\section{Functors associated to Schatten classes }
\label{sec:2}
\subsection{Category of representations of a $C^*$-algebra.}
\label{subsec 1}The above constructions characterize the family of spaces(with measure) for which $S_p(X)$ is non-empty. Studying $S_p(X)$ for these spaces results in some abstract non-sense. We take a short exact sequence of Banach spaces, as defined in \cite{Castillo}. Consider the category of Banach spaces, $\bf Ban_1$. For each $1\leq p<\infty$, we get an exact sequence $E_p$ -
\begin{equation}
\begin{tikzcd}
0 \arrow[r]  & S_p(X)\arrow[r,"\pi"] & S_p(L^2(X))\arrow[r] &\frac{S_p(L^2(X))}{\pi(S_p(X))}\arrow[r] &0 
\end{tikzcd}
\end{equation}
 Also for each pair $(p,q)$ such that $1\leq p<q<\infty$, we have the natural contractive injections - \begin{equation}i_{p,q}:S_p(L^2(X))\to S_q(L^2(X))~~~\text{and }~~~~ \pi^{-1}\circ i_{p,q}\circ \pi:S_p(X)\to S_q(X)\end{equation}
It is easy to verify that $\{E_p\}_{p\geq 1 }$ forms a directed system in the category of exact sequences of Banach spaces (morphism between two exact sequences $E_p$ and $E_q$ for $q>p$, denoted by $\phi_{p,q}$, are required to be contractions between the corresponding nodes) 
\begin{lemma}
	The pair $\{E_p , \phi_{p,q}\}$ forms a directed system on the set $[1,\infty)$ in the category of exact sequences of Banach spaces.
\end{lemma}
\begin{proof}
	This is just an easy verification exercise. Fig. \ref{Fig2} gives us a complete picture. Vertical columns are short exact sequences and each row, $D_0, D_1$ and $D_2$ forms a directed system in the category of Banach spaces. Notice that each rectangular block commutes.
	\begin{figure} 
		\adjustbox{scale=1,center}{%
			\begin{tikzcd}
			& 0                                                   & 0                                                  & 0                                           & 0                                    &     \\
			D_2& \frac{S_1(L^2(X))}{\pi(S_1(X))} \arrow[u] \arrow[r, dashed] & \frac{S_p(L^2(X)}{\pi(S_p(X))} \arrow[u] \arrow[r, dashed] & \frac{B_0(L^2(X))}{\pi(C_0(X))} \arrow[r] \arrow[u] & \frac{B(L^2(X))}{\pi(L^\infty(X))} \arrow[u] \\
			D_1  & S_1(L^2(X)) \arrow[u] \arrow[r, "i", dashed]        & S_p(L^2(X)) \arrow[r, "i", dashed] \arrow[u]       & B_0(L^2(X)) \arrow[u] \arrow[r, "i"]        & B(L^2(X)) \arrow[u]                   \\
			D_0  & S_1(X) \arrow[u, "\pi"] \arrow[r,"\pi^{-1}i_{1,p}\pi", dashed]           & S_p(X) \arrow[u, "\pi"] \arrow[r,"", dashed]          & C_0(X) \arrow[u, "\pi"] \arrow[r]           & L^\infty(X) \arrow[u, "\pi"]         \\
			& 0 \arrow[u]                                         & 0 \arrow[u]                                        & 0 \arrow[u]                                 & 0 \arrow[u]                          &     \\
			& E_1 \arrow[r, dashed]                               & E_p \arrow[r, dashed]                              & E_0 \arrow[r]                               & E_\infty                             &    
			\end{tikzcd}
		}	\caption{}\label{Fig2}
	\end{figure}
\end{proof}
 As we mentioned in remark \ref{Rem1}, the definition of $S_p(X)$ depends on the left regular representation of $L^\infty(X)$. Therefore, it leads us to study the more general case. Let $\mathcal A$ be a $C^*$-algebra, and $\pi:\mathcal A\to B(\mathcal H_\pi)$ be a representation(not necessarily faithful) of $\mathcal A$. We define $$S_p^\pi(\mathcal A)=\pi^{-1}\left(S_p(\mathcal H_\pi)\right) .$$
Now, we have an exact sequence $E_p^\pi$, and for $p>q$ we have the following morphism between $E_p^\pi\to E_q^\pi$ arising through natural maps.
\begin{equation}
\begin{tikzcd}
0 \arrow[r] & \frac{S_p^\pi(\mathcal A)}{\operatorname{ker}(\pi)} \arrow[r] \arrow[d] & S_p(\mathcal H_\pi) \arrow[r] \arrow[d] & \frac{S_p(\mathcal H_\pi)}{\pi(S_p^\pi(\mathcal A))} \arrow[r] \arrow[d] & 0 & E_p^\pi  \\
0 \arrow[r] & \frac{S_q^\pi(\mathcal A)}{\operatorname{ker}(\pi)} \arrow[r]           & S_q(\mathcal H_\pi) \arrow[r]           & \frac{S_q(\mathcal H_\pi)}{\pi(S_q^\pi(\mathcal A))} \arrow[r]           & 0 & E_q^\pi 
\end{tikzcd}
\end{equation}
\begin{theorem}\label{Th.3}
	Schatten classes on Hilbert spaces gives a functor from the category $\mathcal Rep(\mathcal A)$ of faithful representations of $C^*$-algebra $\mathcal A$ to the category $DS_{[1,\infty)}$ of directed systems on the directed set $\{p~:1\leq p< \infty\}$.
\end{theorem}
\begin{proof}
	The assignment $\mathcal S: \mathcal Rep(\mathcal A)\to DS_{[1,\infty)}$ given by $$S(\pi)=E^\pi$$ is a well-defined functor, where $E^\pi$ is the directed system $\{E_p^\pi\}_{p\geq 1}$
	
$$E_1^\pi\to\hdots E_p^\pi\to\hdots $$
\end{proof}
\subsection{Category of group representations.}
\label{subsec 2}
Let $G$ be a locally compact Hausdorff group. Consider the unitary representation $\pi:G\to U(\mathcal H_\pi)$ of $G$ on a Hilbert space $\mathcal H_\pi$. Consider the induced non-degenerate representation  $\pi:L^1(G)\to B(\mathcal H_\pi)$ of Banach-* algebra $L^1(G)$ \cite[3.2]{Folland}, which is given by  $$\left<\pi(f)u,v\right>=\int_G f(x)\left<\pi(x)u,v\right>dx .$$ Now, we define \begin{equation}
T_p^\pi (G)=\pi^{-1}(S_p(\mathcal H_\pi))
\end{equation}
the pullback (through $\pi$) of Schatten p-class operators on $\mathcal H_\pi$. Clearly, $T_p^\pi(G)$ is an ideal in $L^1(G)$.\\
With this notion, we get a functor associated Schatten classes as described in next theorem.

\begin{theorem}
	Let $G$ be a locally compact Hausdorff group. Associated to Schatten classes, there is a functor $\mathcal S :\mathcal Rep(G) \to DS_{[1,\infty)}$.
\end{theorem} 
\begin{proof}
	This is again an easy verification that  $\pi\to \{E_p^\pi\}_{p\geq 1}$ is a functor.
	Where $E_p^\pi$ is the exact sequence given by 
\begin{equation}	\begin{tikzcd}
	0 \arrow[r] & \frac{T_p^\pi(G)}{\operatorname{ker}(\pi)} \arrow[r] & S_p(\mathcal H_\pi) \arrow[r] & \frac{S_p(\mathcal H_\pi)}{\pi(T_p^\pi(G))} \arrow[r] & 0
	\end{tikzcd}
	\end{equation}
	and the morphisms are again the natural ones similar to the ones  considered in Theorem \ref{Th.3}.
	\end{proof}
 
\subsection{Some problems}
We end this note with a few problems.
\begin{enumerate}
	\item Can we give a condition on topology of $G$, similar to the condition on measure on $X$ in Theorem \ref{Th.1}, for the left regular representation $\pi$ of a locally compact Hausdorff group $G$ such that the $\pi^{-1}(S_p( L^2(G)))$ is non-empty?
	\item  It is shown in \cite[Th. 4.1]{Castillo1} that direct limits of a directed system of Banach spaces exist in the category $Ban_1$.  Its also been shown in the remark after \cite[Def. 2.4]{Kobi} that the filtered limits  may not be exact. Does filtered limit from Fig. \ref{Fig2} $$0\to \underset{\longrightarrow}{D_0}\to \underset{\longrightarrow}{D_1}\to \underset{\longrightarrow}{D_2}\to 0$$ preserve the exactness? If yes then what about the same problem with arbitrary representation $\pi$?
\end{enumerate}
\begin{acknowledgements}
We would like to acknowledge the financial support provided by C.S.I.R and U.G.C in form of Junior Research Fellowship. 
\end{acknowledgements}

%
%



\end{document}